\documentclass[a4paper,10pt]{amsart}
\usepackage{amsmath,amsthm,amssymb,enumerate}
\usepackage{epsfig}
\usepackage[latin1]{inputenc}
\usepackage{bm}
\usepackage{bbm}
\usepackage{esint}
\usepackage{amsfonts}
\usepackage{amsxtra}
\usepackage{euscript,mathrsfs}
\usepackage{color}
\usepackage[left=3.8cm,right=3.8cm,top=3.5cm,bottom=3cm]{geometry}
\usepackage{hyperref}
\hypersetup{colorlinks=true, linkcolor=red!70!black, citecolor=green!50!black}
\allowdisplaybreaks
\normalsize

\usepackage{tikz-cd}
\usetikzlibrary{decorations.pathreplacing}

\numberwithin{equation}{section}

\newtheorem{theorem}{Theorem}[section]
\newtheorem{lemma}[theorem]{Lemma}
\newtheorem{proposition}[theorem]{Proposition}

\newtheorem{remark}[theorem]{Remark}

\newtheorem{openquestion}[theorem]{Open Question}

\numberwithin{equation}{section}

\newcommand{\norm}[1]{\left\lVert#1\right\rVert}

\def\XXint#1#2#3{{\setbox0=\hbox{$#1{#2#3}{\int}$ }
\vcenter{\hbox{$#2#3$ }}\kern-.6\wd0}}

\def\dashint{\Xint-}

\newcommand{\bv}{\operatorname{BV}}
\renewcommand{\geq}{\geqslant}

\newcommand{\di}{\operatorname{div}}

\newcommand{\dif}{\operatorname{d}\!}

\newcommand{\R}{\mathbb{R}}

\newcommand{\locc}{\operatorname{loc}}

\newcommand{\sym}{\operatorname{sym}}

\newcommand{\ball}{\operatorname{B}}

\newcommand{\sobo}{\operatorname{W}}
\newcommand{\lebe}{\operatorname{L}}

\newcommand{\hold}{\operatorname{C}}

\newcommand{\curl}{\operatorname{curl}}
\renewcommand{\leq}{\leqslant}

\newcommand{\sg}{\varepsilon}
\newcommand{\bmo}{\operatorname{BMO}}

\newcommand{\A}{\mathbb{A}}

\newcommand{\besov}{\operatorname{B}}

\newcommand{\scew}{\mathrm{skew}}

\renewcommand{\dashint}{\fint}

\begin{document}

\title[Korn-Maxwell-Sobolev Inequalities]{On Korn-Maxwell-Sobolev Inequalities}

\thanks{}
\author[Gmeineder]{Franz Gmeineder}
\address{F.G.: Mathematisches Institut der Universit\"{a}t Bonn, Endenicher Allee 60, 53115 Bonn, Germany. E-mail: \texttt{fgmeined@math.uni-bonn.de}}
\author[Spector]{Daniel Spector}
\address{D.S.: Okinawa Institute of Science and Technology Graduate University, 1919-1 Tancha, Onna-son, Kunigami-gun, Okinawa, Japan 904-0495. E-mail: \texttt{daniel.spector@oist.jp}}
\date{\today}
\thanks{\emph{Acknowledgments.} The first author is grateful to the Hausdorff Centre of Mathematics, Bonn, for financial support. He is moreover indebted to the Okinawa Institute of Science and Technology for financial support that made possible a visit in March 2020, where large parts of this paper were finished.}
\maketitle

\begin{abstract}
We establish a family of inequalities that allow one to estimate the $\lebe^{q}$-norm of a matrix-valued field by the $\lebe^{q}$-norm of an elliptic part and the $\lebe^{p}$-norm of the matrix-valued curl. This particularly extends previous work by \textsc{Neff} et al. and, as a main novelty, is applicable in the regime $p=1$. 
\end{abstract}

\section{Introduction}
\subsection{Korn-Maxwell-type inequalities}
Coercive or \textsc{Korn}-type inequalities are the key ingredient for the treatment of a variety of problems from elasticity or fluid mechanics \cite{Friedrichs,FuchsSeregin,Mosolov}. In its most basic form, the classical Korn inequality asserts that for each $1<p<\infty$ there exists a constant $c=c(p)>0$ such that for all $u\in\hold_{c}^{\infty}(\R^{3};\R^{3})$ there holds 
\begin{align}\label{eq:Korn}
\|Du\|_{\lebe^{p}(\R^{3})}\leq c \|\sg(u)\|_{\lebe^{p}(\R^{3})}, 
\end{align}
where $\sg(u):=\frac{1}{2}(Du+Du^{\top})$ is the symmetric gradient of $u$. As $\sg(u)$ is in general a weaker quantity than the full gradient $Du$, inequalities such as \eqref{eq:Korn} are non-trivial and usually arise as a consequence of singular integral estimates. The latter necessitates the exponent restriction $1<p<\infty$ as \eqref{eq:Korn} fails to hold for $p=1$ by a celebrated counterexample due to \textsc{Ornstein} \cite{Ornstein}.

There are several ways to generalise inequalities of the form \eqref{eq:Korn}. One possibility to do so are the so-called \emph{Korn-Maxwell inequalities} that arise, for example, in the mathematical theory of elasticity or plasticity, respectively. Such inequalities have been considered and studied extensively by \textsc{Neff} and coauthors, \textit{cf.} \cite{Neff3,Neff4,Neff1,Neff}; one form thereof is given by 
\begin{align}\label{eq:KornMaxwellIntro}
\|F\|_{\lebe^{p}(\Omega)}\leq c(\|F^{\sym}\|_{\lebe^{p}(\Omega)}+\|\curl(F)\|_{\lebe^{p}(\Omega)})\;\text{for}\;F\in\hold_{c}^{\infty}(\Omega;\R^{3\times 3})
\end{align}
for open and bounded sets $\Omega\subset\R^{3}$. Here, $F^{\sym}:=\frac{1}{2}(F+F^{\top})$ is the symmetric part of the $\R^{3\times 3}$-valued map $F$ and $\curl(F)$ denotes its row-wise curl. Let us note that if $F\in\hold_{c}^{\infty}(\Omega;\R^{3\times 3})$ with $\curl(F)=0$, then $F$ is a gradient, and so \eqref{eq:KornMaxwellIntro} yields \eqref{eq:Korn}. Based on the zero boundary values of the admissible maps, we refer to \eqref{eq:KornMaxwellIntro} as \emph{Korn-Maxwell inequality of the first kind}. Another relevant variant of \eqref{eq:KornMaxwellIntro} is given by replacing the symmetric part of $F$ on the right-hand side of \eqref{eq:KornMaxwellIntro} by its trace-free or deviatoric part $F^{\mathrm{dev}}:=F^{\sym}-\frac{1}{3}\mathrm{tr}(F)\mathbbm{1}_{3}$, where $\mathbbm{1}_{3}$ denotes the $(3\times 3)$-unit matrix. As established in \cite{Neff4}, the analogue of \eqref{eq:KornMaxwellIntro} reads
\begin{align}\label{eq:KornMaxwellIntroA}
\|F\|_{\lebe^{p}(\Omega)}\leq c(\|F^{\mathrm{dev}}\|_{\lebe^{p}(\Omega)}+\|\curl(F)\|_{\lebe^{p}(\Omega)})\;\text{for}\;F\in\hold_{c}^{\infty}(\Omega;\R^{3\times 3})
\end{align}
for open and bounded sets $\Omega\subset\R^{3}$. The aim of this paper is to provide a common gateway to inequalities of the form \eqref{eq:KornMaxwellIntro} and their natural modifications, the Korn-Maxwell-Sobolev inequalities. This particularly motivates a framework  that is both applicable to more general operators than the symmetric or trace-free gradient \emph{and to the borderline case} $p=1$, a theme that shall now be described in detail.

\subsection{Korn-Maxwell-Sobolev-type inequalities}
We start our discussion by noting that both \eqref{eq:KornMaxwellIntro} and \eqref{eq:KornMaxwellIntroA} cannot hold for $p=1$. This can be seen by taking $F$ to be gradients, $F=\nabla u$, and recalling that e.g. \eqref{eq:Korn} does not extend to $p=1$ by \textsc{Ornstein}'s Non-Inequality \cite{Ornstein,KirchheimKristensen}.  On the other hand, both inequalities \eqref{eq:KornMaxwellIntro} and \eqref{eq:KornMaxwellIntroA} do not involve the requisite exponents that admit suitable scaling. For $1\leq p < 3$ instead, we consider the following inequality: 
\begin{align}\label{eq:SKM1}
\|F\|_{\lebe^{p^{*}}(\Omega)}\leq c(\|F^{\sym}\|_{\lebe^{p^{*}}(\Omega)}+\|\curl(F)\|_{\lebe^{p}(\Omega)})\;\;\;\text{for}\;F\in\hold_{c}^{\infty}(\Omega;\R^{3\times 3}), 
\end{align}
where $p^{*}:=\frac{3p}{3-p}$ is the Sobolev conjugate of $p$. Clearly, if $p=3$ or $p>3$, the $\lebe^{p^{*}}$-norms should be replaced by the $\mathrm{BMO}$- or corresponding $\hold^{0,1-3/p}$-H\"{o}lder (semi)norms, respectively. Postponing the incorporation of other function spaces, we now set up the framework for the main results of the paper. To this end, let $\A$ be a linear, homogeneous differential operator of order one on $\R^{3}$ from $\R^{3}$ to some $\R^{N}$. By this we understand that $\A$ has a representation 
\begin{align}\label{eq:form}
\A u :=\sum_{j\in\{1,2,3\}}\A_{j}\partial_{j}u,\qquad u\colon\R^{3}\to\R^{3}
\end{align}
with fixed linear maps $\A_{j}\colon\R^{3}\to\R^{N}$. 
Following \textsc{H\"{o}rmander} and \textsc{Spencer}  \cite{Hoermander,Spencer}, we call $\A$ \emph{elliptic} provided for each $\xi\in\R^{3}\setminus\{0\}$ the symbol map 
\begin{align}\label{eq:elliptic}
\A[\xi]:=\sum_{j\in\{1,2,3\}}\xi_{j}\A_{j} \colon \R^{3}\to\R^{N}
\end{align}
is injective. Adopting the viewpoint of \cite{GmeinederRaita,Raita}, $\A u = A[\nabla u]$ for some linear map $A\in\mathscr{L}(\R^{3\times 3};\R^{N})$, and we call $A$ the \emph{matrix representative} of $\A$; equally, every $A\in\mathscr{L}(\R^{3\times 3};\R^{N})$ induces an operator $\A$ by means of $\A u:=A[\nabla u]$. Our main theorem then is this: 
\begin{theorem}[Korn-Maxwell-Sobolev I]\label{thm:main}
Let $1\leq p<3$ and $N\in\mathbb{N}$. Then the following are equivalent for $A\in\mathscr{L}(\R^{3\times 3};\R^{N})$: 
\begin{enumerate}
\item\label{item:thmmainA} $A$ induces an elliptic operator $\A$ of the form \eqref{eq:form}. 
\item\label{item:thmmainB} There exists a constant $c=c(p,A)>0$ such that 
\begin{align}\label{eq:KornMaxwellAA}
\|F\|_{\lebe^{p^{*}}(\Omega)}\leq c(\|A[F]\|_{\lebe^{p^{*}}(\Omega)}+\|\curl(F)\|_{\lebe^{p}(\Omega)})
\end{align}
holds for all open sets $\Omega\subset\R^{n}$ and $F\in\hold_{c}^{\infty}(\Omega;\R^{3\times 3})$. 
\end{enumerate} 
\end{theorem} 
In the case where $1<p<\infty$ and $\Omega\subset\R^{3}$ is bounded, the same method underlying the proof of the previous theorem yields that the estimate 
\begin{align}\label{eq:suboptimal}
\|F\|_{\lebe^{p}(\Omega)}\leq c(\|A[F]\|_{\lebe^{p}(\Omega)}+\|\curl(F)\|_{\lebe^{p}(\Omega)})\qquad\text{for}\;F\in\hold_{c}^{\infty}(\Omega;\R^{3\times 3})
\end{align}
is equivalent to $A$ inducing an elliptic operator $\A$ of the form \eqref{eq:form}, see the discussion at the end of Section~\ref{sec:main1}. By smooth approximation, this gives us back the corresponding inequalities considered in \cite{Neff3,Neff4}. 

The proof of Theorem~\ref{thm:main} together with various extensions is provided in Section~\ref{sec:main}. The paper then is  concluded by discussing Korn-Maxwell-Sobolev-type variants of \eqref{eq:KornMaxwellIntro} and \eqref{eq:KornMaxwellIntroA} on cubes in the situation of non-zero boundary values in Section~\ref{sec:nonzero} for the particularly imporant case of the symmetric and trace-free symmetric gradient operators. Lastly, the appendix provides an extension theorem for divergence-free vector fields.

Let us note that, when preparing the final version of the manuscript, we became aware of the recent preprint \cite{ContiGarroni} of \textsc{Conti \& Garroni} which also uses the Bourgain-Brezis estimate for solenoidal vector fields to arrive at a special case of Theorems~\ref{thm:main} and~\ref{thm:main1} for the symmetric gradient operator $\A=\frac{1}{2}(D+D^{\top})$, cf. \cite[Thm.~1.2]{ContiGarroni}. Whereas \cite{ContiGarroni} focuses on the symmetric gradient and a $\mathrm{SO}(n)$-rigidity result \cite[Thm.~1.1]{ContiGarroni}, our paper rather concentrates on a characterization of operators $\A$ to yield such inequalities (Section~\ref{sec:main}), leading to a unifying theory for all $1\leq p<3$; also note that Section~\ref{sec:nonzero} proceeds slightly differently and covers the trace-free symmetric gradient as well.
\section*{Notation}
For a vector field $F\colon \R^{3}\to\R^{3\times 3}$, we denote $F^{j}$ the $j$-th row of $F$, $j\in\{1,2,3\}$. For vectors $a,b\in\R^{n}$, we denote $\langle a,b\rangle$ the euclidean inner product on $\R^{n}$, and for matrices $A=(a_{ij}),B=(b_{ij})\in\R^{n\times n}$, we use the notation $\langle A,B\rangle =\sum_{ij}a_{ij}b_{ij}$; the meaning will be clear from the context. The $n$-dimensional Lebesgue measure will be denoted $\mathscr{L}^{n}$, and the symbol $\mathscr{F}$ represents the Fourier transform as usual.
\section{Proof of Theorem~\ref{thm:main}}\label{sec:main}
\subsection{Korn-Maxwell inequality of the first kind}\label{sec:main1}
In this section we establish Theorem~\ref{thm:main}. As a vital ingredient, we require 
\begin{lemma}[of Mihlin-H\"{o}rmander type {\cite[Thm.~4.13]{Duo}}]\label{lem:MH}
Let $m\in\hold^{\infty}(\R^{n}\setminus\{0\})$ be a function that is homogeneous of degree zero and $T_{m}$ the operator given by $(T_{m}f)^{\widehat{}}=m\widehat{f}$, then there exist $a\in\mathbb{R}$ and $\Theta\in\hold^{\infty}(\mathbb{S}^{n-1})$ with zero average over $\mathbb{S}^{n-1}$ such that for any $f\in\hold_{c}^{\infty}(\R^{n})$ there holds (with $\mathrm{p.v.}$ denoting the Cauchy principal value) 
\begin{align*}
T_{m}f = af + \mathrm{p.v.}\frac{\Theta(\frac{\cdot}{|\cdot|})}{|\cdot|^{n}}*f, 
\end{align*}
and hence is an $\lebe^{p}$-bounded operator for any $1<p<\infty$.
\end{lemma}
Moreover, we recall that for $0<s<n$ and $f\in\lebe_{\locc}^{1}(\R^{n})$, the \emph{Riesz potential} of order $s$ is defined by 
\begin{align*}
I_{s}f(x):=\frac{1}{\gamma(s)} \int_{\R^{n}}\frac{f(y)}{|x-y|^{n-s}}\dif y,\qquad x\in\R^{n},
\end{align*} 
where 
\begin{align*}
\gamma(s):= \frac{ \pi^{n/2} 2^s \Gamma \left(\frac{s}{2}\right)}{\Gamma\left(\frac{n}{2}-\frac{s}{2}\right)}.
\end{align*}

\begin{proof}
We may suppose that $\Omega=\R^{3}$, otherwise we extend $F$ by zero to the entire $\R^{3}$. Ad~$\ref{item:thmmainA}\Rightarrow\ref{item:thmmainB}$.  Let $1\leq p <3$. Writing $F=(F^{1},F^{2},F^{3})^{\top}$ with $F^{j}\in\hold_{c}^{\infty}(\R^{3};\R^{3})$, for each $j\in\{1,2,3\}$, we apply the Helmholtz decomposition to $F^{j}$ to obtain $F^{j}=F_{\di}^{j}+F_{\curl}^{j}$, where $F_{\di}^{j}$ is the divergence-free and $F_{\curl}^{j}$ the $\curl$-free part of $F^{j}$. It is well-known that $F_{\di}^{j}$ and $F_{\curl}^{j}$ can be obtained from $F$ by means of 
\begin{align}\label{eq:Helmholtzrep}
\begin{split}
&F_{\di}^{j}(x) = \frac{1}{4\pi}\mathrm{curl}_{x}\int_{\R^{3}}\frac{\curl(F^{j}(y))}{|x-y|}\dif y, \\ 
&F_{\curl}^{j}(x) = -\frac{1}{4\pi}\nabla_{x}\int_{\R^{3}}\frac{\di(F^{j}(y))}{|x-y|}\dif y.
\end{split}
\end{align}
We put $F_{\di}:=(F_{\di}^{1},F_{\di}^{2},F_{\di}^{3})^{\top}$ and $F_{\curl}:=(F_{\curl}^{1},F_{\curl}^{2},F_{\curl}^{3})^{\top}$. With the Helmholtz decomposition and the representation \eqref{eq:Helmholtzrep}, we have with $p^{*}=\frac{3p}{3-p}$
\begin{align}\label{eq:proofbound1}
\begin{split}
\|F\|_{\lebe^{p^{*}}(\R^{3})} & \leq \|F_{\di}\|_{\lebe^{p^{*}}(\R^{3})}+\|F_{\curl}\|_{\lebe^{p^{*}}(\R^{3})}=: \mathrm{I} + \mathrm{II},
\end{split}
\end{align}
and depending on $p$, the single terms are treated differently as follows. 

Ad~$\mathrm{I}$, \emph{Case $1<p<3$}. By the fractional integration theorem (see, e.g. \cite[Theorem 1 on p.~119]{Stein_1970}) we have that if $p>1$ and $s>0$ satisfy $1<sp<n$, then $I_{s}\colon\lebe^{p}(\R^{n})\to\lebe^{\frac{np}{n-sp}}(\R^{n})$ boundedly. Therefore, by $\eqref{eq:Helmholtzrep}_{1}$ and the fractional integration theorem with $s=1$ and $n=3$, we consequently have with $c=c(p)>0$
\begin{align}\label{eq:IestA}
\begin{split}
\mathrm{I} & \leq \sum_{j\in\{1,2,3\}}\frac{1}{4\pi}\norm{\nabla \Big( \frac{1}{|\cdot|}*\curl(F^{j})\Big)}_{\lebe^{p^{*}}(\R^{3})}\\ 
& \leq c\sum_{j\in\{1,2,3\}}\norm{ \frac{1}{|\cdot|^{2}}*\curl(F^{j})}_{\lebe^{p^{*}}(\R^{3})}\\
& \leq c\sum_{j\in\{1,2,3\}}\|\curl(F^{j})\|_{\lebe^{p}(\R^{3})} \leq c\|\curl(F)\|_{\lebe^{p}(\R^{3})}. 
\end{split}
\end{align}
Ad~$\mathrm{I}$, \emph{Case $p=1$}. It is well-known that the fractional integration theorem does not extend to $p=1$ for general functions.  However, the case we treat has the additional information that
\begin{align}\label{divergence_free}
\operatorname*{div} \curl(F^{j})=0
\end{align}
in the sense of distributions for $j\in\{1,2,3\}$.  Therefore we may use the Bourgain-Brezis estimate $\|T\|_{\lebe^{3/2}(\R^{3})}\leq c\|\curl T\|_{\lebe^{1}(\R^{3})}$ for divergence-free $T\in\lebe_{\locc}^{1}(\R^{3};\R^{3})$ (cf. \cite[Thm.~2]{BB04A}) or the inequality (1.3) on p.~294 of \cite{SSVS}, respectively, in place of the classical theorem on fractional integration to conclude the desired result with the rest of the argument unchanged.  



We now come to the estimation of $\mathrm{II}$. By ellipticity of $\A$, \emph{cf.} \eqref{eq:elliptic}, for $i\in\{1,2,3\}$ the Fourier multiplication operator 
\begin{align}\label{SIO}
T_{\A}^{i}\colon \psi \mapsto \mathscr{F}^{-1}[\xi_{i}(\A^{*}[\xi]\A[\xi])^{-1}\A^{*}[\xi]\widehat{\psi}(\xi)]
\end{align}
satisfies $T_{\A}^{i}(\A\psi)=\partial_{i}\psi$ for $\psi\in\hold_{c}^{\infty}(\R^{3};\R^{3})$. The symbol map $\R^{3}\setminus\{0\}\ni\xi\mapsto \xi_{i}(\A^{*}[\xi]\A[\xi])^{-1}\A^{*}[\xi]\in \mathscr{L}(\R^{N};\R^{3})$ is of class $\hold^{\infty}(\R^{3}\setminus\{0\};\mathscr{L}(\R^{N};\R^{3}))$ and homogeneous of degree zero. Lemma~\ref{lem:MH} thus implies that $T_{\A}^{i}$ extends to an $\lebe^{q}$-bounded operator for all $1<q<\infty$. Applying this result componentwise, for each $1<q<\infty$ there exists $c=c(q,\A)>0$ such that 
\begin{align}\label{eq:multiest}
\begin{split}
\|F_{\curl}\|_{\lebe^{q}(\R^{3})} & \stackrel{\eqref{eq:Helmholtzrep}_{2}}{\leq} c\norm{\nabla\Big(\frac{1}{|\cdot|}*\di(F^{l})\Big)_{l=1,2,3}}_{\lebe^{q}(\R^{3})} \\ 
&\;\;\leq c\sum_{i\in\{1,2,3\}}\norm{T_{\A}^{i}\Big(\A\Big(\Big(\frac{1}{|\cdot|}*\di(F^{l})\Big)_{l=1,2,3}\Big)\Big)}_{\lebe^{q}(\R^{3})}\\ 
&\;\;\leq c\norm{\A\Big(\Big(\frac{1}{|\cdot|}*\di(F^{l})\Big)_{l=1,2,3}\Big)}_{\lebe^{q}(\R^{3})}\\
&\;\;\leq c\norm{A\left[\nabla \Big(\frac{1}{|\cdot|}*\di(F^{l})\Big)_{l=1,2,3}\right]}_{\lebe^{q}(\R^{3})}\\
&\;\; \leq c \norm{A[F_{\curl}]}_{\lebe^{q}(\R^{3})} \leq c\norm{A[F]}_{\lebe^{q}(\R^{3})},
\end{split}
\end{align}
since the entries of $A[F^{\curl}]$ are linear combinations of the entries of $A[F]$, and~\ref{item:main1B} follows with $q=\frac{3p}{3-p}$.

Ad~'$\ref{item:thmmainB}\Rightarrow\ref{item:thmmainA}$'. This is a standard construction which we review only briefly: Suppose that \ref{item:thmmainB} holds. Applying \eqref{eq:KornMaxwellAA} to $F=\nabla u$ for $u\in\hold_{c}^{\infty}(\R^{3};\R^{3})$, \eqref{eq:KornMaxwellAA} implies $\|\nabla u\|_{\lebe^{p^{*}}(\R^{3})}\leq c\|\A u\|_{\lebe^{p^{*}}(\R^{3})}$. Then a classical construction\footnote{Namely, there exist $v\in\R^{3}\setminus\{0\}$ and $\xi\in\R^{3}\setminus\{0\}$ such that $\A[\xi]v=0$. Then put $\psi_{k}(x):=\rho(x)\eta_{k}(\langle x,\xi\rangle)v$ for some $\rho\in\hold_{c}^{\infty}(\ball(0,1);[0,1])$ with $\mathbbm{1}_{\ball(0,1)}\leq\rho\leq\mathbbm{1}_{\ball(0,2)}$ and $\eta_{k}\in\hold_{c}^{\infty}(\R)$ with $\sup_{k}\|\eta_{k}\|_{\lebe^{p}(\R)}<\infty$ and $\lim_{k\to\infty}\|\eta'_{k}\|_{\lebe^{p}(\R)}=\infty$. Then $\sup_{k}\|\A\psi_{k}\|_{\lebe^{p}(\R^{3})}<\infty$ but $\sup_{k}\|D\psi_{k}\|_{\lebe^{p}(\R^{3})}=\infty$.} (see \cite[Prop.~4.1]{ContiGmeineder} for the precise argument in a more general context) implies that $\A$ must be elliptic.  The proof is complete.
\end{proof} 
If $\Omega\subset\R^{3}$ is open and bounded, H\"{o}lder's inequality implies that $\|I_{1}f\|_{\lebe^{q}(\Omega)}\leq c(q,\mathrm{diam}(\Omega))\|f\|_{\lebe^{q}(\Omega)}$ for $f\in\hold_{c}^{\infty}(\Omega)$ for any $1\leq q < \infty$. With this estimate instead of \eqref{eq:IestA},  \eqref{eq:suboptimal} follows; however, note that this estimate does not scale conveniently.

\subsection{Variations on Korn-Maxwell-Sobolev-type inequalities}
We conclude this section by discussing several other embeddings. The underlying approach is the same as for Theorem~\ref{thm:main}, now invoking boundedness properties of fractional and singular integral operators on different function spaces. Suppose that item~\ref{item:main1A} of Theorem~\ref{thm:main} holds.

\emph{Limiting Korn-Maxwell-Sobolev.} If $p=3$ in Theorem~\ref{thm:main}, we let $F\in\hold_{c}^{\infty}(\R^{3};\R^{3\times 3})$ and proceed exactly as up to \eqref{eq:proofbound1}, where $\lebe^{p^{*}}(\R^{3};\R^{3\times 3})$ is now replaced by the space  $\mathrm{BMO}(\R^{3};\R^{3\times 3})$. Since 
\begin{align}
I_{1}\colon \lebe^{n}(\R^{n})\to\mathrm{BMO}(\R^{n})\qquad\text{boundedly}, 
\end{align}
the analogue of \eqref{eq:IestA} yields (upon redefining $\mathrm{I}$ in the obvious way) $\mathrm{I}\leq c\|\curl(F)\|_{\lebe^{3}(\R^{3})}$. On the other hand, the singular integral operators underlying \eqref{SIO} map $\mathrm{BMO}(\R^{3};\R^{N})\to\mathrm{BMO}(\R^{3};\R^{3})$ boundedly (see, e.g. \cite[Theorem 1.1 on p.~296]{Peetre} or \cite[Chpt.~IV.6.3(b)]{Stein}). Thus we obtain 
\begin{align}
\|F\|_{\mathrm{BMO}(\R^{3})}\leq c\big(\|A[F]\|_{\mathrm{BMO}(\R^{3})}+\|\curl(F)\|_{\lebe^{3}(\R^{3})} \big)
\end{align}
for all $F\in\hold_{c}^{\infty}(\R^{3};\R^{3\times 3})$, where $c>0$ is a constant. Here, we have set 
\begin{align*}
\|u\|_{\bmo(\R^{n})}:=\sup_{\substack{Q\,\text{non-degenerate cube}}}\dashint_{Q}\left|u-\dashint_{Q}u\dif y\right\vert\dif x.
\end{align*}

\emph{Korn-Maxwell-Morrey.} Now suppose that $3<p<\infty$. Then it is well-known that there exists a constant $c=c(p)>0$ such that 
\begin{align*}
\|I_{1}f\|_{\dot \hold^{0,\alpha}(\R^{3})}\leq c\|f\|_{\lebe^{p}(\R^{3})}\qquad\text{for all}\;f\in\hold_{c}^{\infty}(\R^{3}), 
\end{align*}
where $\alpha=1-\frac{3}{p}$ and $\|\cdot\|_{{\dot\hold}{^{0,\alpha}}(\R^{3})}$ is the corresponding $\alpha$-H\"{o}lder seminorm. The singular integral operator underlying \eqref{SIO} is bounded as a map ${\dot \hold}{^{0,\alpha}}(\R^{3};\R^{N})\to {\dot\hold}{^{0,\alpha}}(\R^{3};\R^{3})$; this can be seen by the same argument as in \eqref{eq:CZH1}ff. below, realising the H\"{o}lder spaces as Besov spaces (as $0<\alpha<1$) and appealing to \cite[Cor.~6.7.2]{Grafakos}. We thus obtain the estimate 
\begin{align}\label{eq:Morrey}
\|F\|_{\dot \hold^{0,1-\frac{3}{p}}(\R^{3})}\leq c\big(\|A[F]\|_{\dot \hold^{0,1-\frac{3}{p}}(\R^{3})}+\|\curl(F)\|_{\lebe^{p}(\R^{3})} \big)
\end{align}
for all $F\in\hold_{c}^{\infty}(\R^{3};\R^{3\times 3})$, where $c>0$ is a constant. 

\emph{Korn-Maxwell-Lorentz.} Let $1\leq p < \infty$ and $1\leq q \leq \infty$. Recall that for $u\in\lebe_{\locc}^{1}(\R^{n})$ its $(p,q)$-Lorentz norm is given for $1<q<\infty$ by
\begin{align*}
\|u\|_{\lebe^{p,q}(\R^{n})}:=p^{\frac{1}{q}}\Big(\int_{0}^{\infty}t^{q}\mathscr{L}^{n}(\{|u|\geq t\})^{\frac{q}{p}}\frac{\dif t}{t}\Big)^{\frac{1}{q}}
\end{align*}
whereas $\|u\|_{\lebe^{p,\infty}(\R^{n})}:=\sup_{t>0}t\mathscr{L}^{n}(\{|u|\geq t\})^{\frac{1}{p}}$. Then the operator defined in \eqref{SIO} maps $\lebe^{r}\to\lebe^{r}$ boundedly for each $1<r<\infty$, and $\lebe^{1}\to\lebe^{1,\infty}$ boundedly. Hence, by interpolation (see \cite{Hunt}) the operator defined in \eqref{SIO} extends to a bounded linear operator on Lorentz spaces $\lebe^{p,q}(\R^{n})$ for $1<p<\infty$ and all $1\leq q<\infty$. Now, given $1<p<3$, put $p^{*}:=3p/(3-p)$. Then \textsc{O'Neil}'s classical convolution inequality implies that $I_{1}\colon \lebe^{p}(\R^{3})\to \lebe^{p^{*},q}(\R^{3})$ boundedly for any $q\in [1,\infty]$. Then we obtain as above 
\begin{align*} 
\|F\|_{\lebe^{p^{*},q}(\R^{3})} & \leq c\Big(\|A[F]\|_{\lebe^{p^{*},q}(\R^{3})} + \|\curl(F)\|_{\lebe^{p,q}(\R^{3})} \Big)\qquad\text{for all}\;F\in\hold_{c}^{\infty}(\R^{3};\R^{3\times 3}). 
\end{align*}
This estimate equally persists for $p=1$ but must be approached differently. Namely, taking into account \eqref{divergence_free}, by an estimate of  \textsc{Hernandez} and the second named author \cite[Theorem 1.1]{HS} one has the inequality
\begin{align*}
\|I_1\curl(F)\|_{\lebe^{3/2,1}(\R^{3})} \leq C \|\curl(F)\|_{\lebe^{1}(\R^{3})}.
\end{align*}
This completes the argument for the endpoint case $q=1$, while the remaining cases $1<q\leq +\infty$ then follow from a classical inequality due to \textsc{Calder\'on}, 
\begin{align*}
\|g\|_{\lebe^{3/2,q}(\R^{3})} \leq C \|g\|_{\lebe^{3/2,1}(\R^{3})}
\end{align*}
for all $g \in \lebe^{3/2,1}(\R^{3})$.  One could alternatively argue these cases via \textsc{Van Schaftingen}'s duality estimate \cite[Prop.~8.7]{VS11}.

\emph{Fractional Korn-Maxwell.} Let $\theta\in (0,1)$ and $p\in [1,\infty)$. The $\sobo^{\theta,p}$-Sobolev seminorm of a compactly supported function $u\in\lebe_{\locc}^{1}(\R^{n})$ then is given by 
\begin{align*}
\|u\|_{{\dot\sobo}^{\theta,p}(\R^{n})}:=\Big(\iint_{\R^{n}\times\R^{n}}\frac{|u(x)-u(y)|^{p}}{|x-y|^{n+\theta p}}\dif x\dif y\Big)^{\frac{1}{p}}. 
\end{align*}
Given $\theta\in (0,1)$ and $1\leq p < 3$, the desired inequality now takes the form 
\begin{align}\label{eq:fractional}
\|F\|_{{\dot\sobo}^{\theta,p^{*}(\theta)}(\R^{3})}\leq c\big(\|A[F]\|_{{\dot\sobo}^{\theta,p^{*}(\theta)}(\R^{3})}+\|\curl(F)\|_{\lebe^{p}(\R^{3})} \big) 
\end{align}
for $F\in\hold_{c}^{\infty}(\R^{3};\R^{3\times 3})$, where $p^{*}(\theta):=3p/(3-(1-\theta)p)$ denotes the associated Sobolev embedding exponent. Toward \eqref{eq:fractional}, we adopt a slightly more general viewpoint since multiplier theorems are most conveniently stated in terms of homogeneous Besov spaces. Consider the kernel from Lemma~\ref{lem:MH}
\begin{align*}
K(x):=\frac{\Theta(\frac{x}{|x|})}{|x|^{n}},\qquad x\in\R^{n}\setminus\{0\},
\end{align*}
for $\Theta\in\hold^{\infty}(\mathbb{S}^{n-1})$ with zero mean over $\mathbb{S}^{n-1}$.  This kernel satisfies the three Calder\'{o}n-Zygmund-H\"{o}rmander conditions
\begin{align}
& \sup_{0<R<\infty}\frac{1}{R}\int_{\ball(x,R)}|K(x)|\,|x|\dif x \leq A_{1}, \label{eq:CZH1}\\
& \sup_{y\in\R^{n}\setminus\{0\}}\int_{\{x\colon\,|x|\geq 2|y|\}}|K(x-y)-K(x)|\dif y \leq A_{2},\label{eq:CZH2}\\
& \sup_{0<R_{1}<R_{2}<\infty}\left\vert\int_{\{x\colon\,R_{1}<|x|<R_{2}\}}K(x)\dif x\right\vert \leq A_{3}\label{eq:CZH3}
\end{align}
for three finite constants $A_{1},A_{2},A_{3}\geq 0$. Here, \eqref{eq:CZH1} and \eqref{eq:CZH3} straightforwardly follow by passing to polar coordinates, where we moreover use for \eqref{eq:CZH3} that $K$ has vanishing mean over $\mathbb{S}^{n-1}$. By \cite[Prop.~5.2]{Duo} this follows from $|\nabla K(x)|\leq C|x|^{-n-1}$ for all $x\in\R^{n}\setminus\{0\}$ and a constant $C>0$, here being a consequence of $\Theta\in\hold^{\infty}(\R^{n}\setminus\{0\})$. In conclusion, since \eqref{eq:CZH1}--\eqref{eq:CZH3} are satisfied, \cite[Cor.~6.7.2]{Grafakos} implies that $T_{m}$ from Lemma~\ref{lem:MH} is a bounded linear operator on the homogeneous Besov space ${\dot\besov}{_{p,q}^{s}}(\R^{n})$ for all $1\leq p\leq\infty$, $0<q\leq \infty$ and $s\in\R$. By a component-wise application, this carries over to $T_{\A}^{i}$ given by \eqref{SIO} as well. 

Given $1<p<3$ and $\theta\in (0,1)$, put $\overline{p}:=p^{*}(\theta)$ for brevity. Then, e.g., \cite[Chpt.~5.2.3, Thm.~1, Chpt.~2.7.1, Thm.~1(ii)]{Triebel} and \cite[Thm.~2.1]{Jawerth} imply that for any $1<q<\infty$ there holds
\begin{align}\label{eq:triebelinc}
\lebe^{p}(\R^{n})\simeq {\dot{\mathrm{F}}}{_{p,2}^{0}}(\R^{n}) \hookrightarrow {\dot{\mathrm{F}}}{_{\overline{p},q}^{\theta-1}}(\R^{3}) \stackrel{I_{1}}{\longrightarrow} {\dot{\mathrm{F}}}{_{\overline{p},q}^{\theta}}(\R^{3}),
\end{align}
and ${\dot{\mathrm{F}}}{_{\overline{p},q}^{\theta}}(\R^{3})\hookrightarrow{\dot{\mathrm{B}}}{_{\overline{p},q}^{\theta}}(\R^{3})$ boundedly provided $q\geq \overline{p}$ with the corresponding homogeneous Triebel-Lizorkin spaces 
${\dot{\mathrm{F}}}{_{p,q}^{s}}$. For such $\theta,p,q$ we then obtain with the above multiplier estimate
\begin{align}\label{eq:inbetween}
\|F\|_{{\dot{\mathrm{B}}}{_{\overline{p},q}^{\theta}}(\R^{3})} & \leq c\Big(\|A[F]\|_{{\dot{\mathrm{B}}}{_{\overline{p},q}^{\theta}}(\R^{3})} + \|\curl(F)\|_{\lebe^{p}(\R^{3})} \Big),\;\;\;F\in\hold_{c}^{\infty}(\R^{3};\R^{3\times 3}). 
\end{align}
For $p=1$, the requisite modification of \eqref{eq:triebelinc} merely yields \eqref{eq:inbetween} with the homogeneous Hardy-$\mathcal{H}^{1}$-norm of $\curl(F)$ instead of $\|\curl(F)\|_{\lebe^{1}(\R^{3})}$. Yet, validity of \eqref{eq:inbetween} for $p=1$ can be seen as follows: By \textsc{Van Schaftingen}'s duality estimate \cite[Prop.~8.7]{VS11} we have for $\vartheta\in (0,1)$ and $1<p_{2},q_{2}<\infty$ with $\vartheta p_{2} = n$
\begin{align}\label{eq:VSduality1}
\int_{\R^{n}}\langle\Phi,\varphi\rangle \dif x \leq c \|\Phi\|_{\lebe^{1}(\R^{n})}\|\varphi\|_{{\dot{\mathrm{B}}}{_{p_{2},q_{2}}^{\vartheta}}(\R^{n})}
\end{align}
for all $\Phi\in\lebe^{1}(\R^{n};\R^{n})$ with $\mathrm{div}(\Phi)=0$ in $\mathscr{D}'(\R^{n})$ and all $\varphi\in\hold_{c}^{\infty}(\R^{n};\R^{n})$. In consequence, a row-wise application of \eqref{eq:VSduality1} with $\vartheta=1-\theta$ and $p_{2}=\overline{1}'$ yields
\begin{align*}
\|F^{\di}\|_{{\dot{\mathrm{B}}}{_{\overline{1},q}^{\theta}}(\R^{3})} & \leq c\norm{\nabla\Big(\frac{1}{|\cdot|}*\curl(F)\Big)}_{{\dot{\mathrm{B}}}{_{\overline{1},q}^{\theta}}(\R^{3})} \\ 
& \leq c\norm{\curl(F)}_{{\dot{\mathrm{B}}}{_{\overline{1},q}^{\theta-1}}(\R^{3})} \leq c\norm{\curl(F)}_{({\dot{\mathrm{B}}}{_{\overline{1}',q'}^{1-\theta}}(\R^{3}))'}\leq c\|\curl(F)\|_{\lebe^{1}(\R^{3})}. 
\end{align*}
Now \eqref{eq:fractional} follows upon realising that ${\dot\sobo}{^{\theta,\overline{p}}}(\R^{3})\simeq {\dot\besov}{_{\overline{p},\overline{p}}^{\theta}}(\R^{3})$; other variants of \eqref{eq:fractional} involving other Besov spaces can be obtained similarly.
\begin{remark}
For an open set with Lipschitz boundary and $1<p<\infty$, the space $\sobo_{0}^{\curl,p}(\Omega;\R^{3\times 3})$ may be introduced as the completion of $\hold_{c}^{\infty}(\Omega;\R^{3\times 3})$ for the norm $\|u\|_{\curl,p}:=(\|u\|_{\lebe^{p}(\Omega)}^{p}+\|\curl(u)\|_{\lebe^{p}(\Omega)}^{p})^{\frac{1}{p}}$. Such fields can be characterised as those $u\in\lebe^{p}(\Omega;\R^{3\times 3})$ such that $\curl(u)\in\lebe^{p}(\Omega;\R^{3\times 3})$ and the componentwise tangential traces $u^{i}\times\nu_{\partial\Omega}$ vanish in $\besov_{p,p}^{-1/p}(\partial\Omega;\R^{3})$, $i\in\{1,2,3\}$ (see, e.g., \cite[Sec.~3]{Neff3}). By density, all of the previous inequalities on open and bounded Lipschitz domains $\Omega$ persist for such maps.
\end{remark}
\section{Korn-Maxwell-Sobolev inequality of the second kind}\label{sec:nonzero}
We conclude the paper by addressing a variant of the Korn-Maxwell-Sobolev type inequality on domains that allows for non-zero boundary values. Here our focus  is on the specific operators $\sg$ or $\sg^{D}$ as alluded to in the introduction, \emph{cf.} \eqref{eq:KornMaxwellIntro} and \eqref{eq:KornMaxwellIntroA}; the case of general elliptic operators is addressed below in Open Question~\ref{OQ:1}. 

To keep our exposition at a reasonable length, we work with the unit cube $Q:=(0,1)^{3}$ in $\R^{3}$ throughout; see Open Question~\ref{OQ:1} for more general domains. We note that inequalities \eqref{eq:KornMaxwellIntro} and \eqref{eq:KornMaxwellIntroA} do not extend to maps $F\in\hold^{\infty}(\overline{Q};\R^{3\times 3})$.
In fact, should \eqref{eq:KornMaxwellIntro} hold for all $F\in\hold^{\infty}(\overline{Q};\R^{3\times 3})$, we necessarily have 
\begin{align*}
(F^{\sym}\equiv 0\;\text{and}\;\curl(F)=0)\Longrightarrow F \equiv 0\;\text{in}\;Q.
\end{align*}
To see this, $\curl(F)=0$ implies by virtue of $Q$ being simply connected that $F=\nabla u$ for some $u\in\hold^{\infty}(\overline{Q};\R^{3})$, and $F^{\sym}=0$ yields $F^{\sym}=\sg(u)=0$. By connectedness of $Q$, $u$ must be of the form $u(x)=Ax+b$ for some $A\in\R_{\mathrm{skew}}^{3\times 3}$ and some $b\in\R^{3}$; maps of this form are called \emph{rigid deformations} and denoted $\mathcal{R}(\R^{3})$. But then $F=\nabla u = A$ which, in general, does not equal zero. A similar argument also applies to inequalities of the form \eqref{eq:KornMaxwellIntroA}, where we must now use the fact that the nullspace of $\sg^{D}(u):=\sg(u) - \frac{1}{n}\di(u)\mathbbm{1}_{n}$ for $n\geq 3$ is given by the conformal Killing vectors
\begin{align*}
\mathcal{K}(\R^{n}):=\{\mathbf{p}\colon x \mapsto 2\langle a,x\rangle x - |x|^{2}a+Q'x+\rho x + b\colon\;a,b\in\R^{n},\;\rho\in\R,\;Q'\in\R_{\mathrm{skew}}^{n\times n}\}, 
\end{align*}
see \cite{Reshetnyak}. In light of these considerations, the inequality of interest consequently is given by the following
\begin{theorem}[Korn-Maxwell-Sobolev II]\label{thm:main1}
Let $1\leq p <3$. Then there exists a constant $c=c(p)>0$ such that the following hold:
\begin{enumerate}
\item\label{item:main1A} For all $F\in\hold^{\infty}(\overline{Q};\R^{3\times 3})$ with 
\begin{align}\label{eq:FIX1}
\int_{Q}\langle F,\Pi\rangle\dif x = 0\qquad\text{for all}\;\Pi\in\R_{\scew}^{3\times 3} = \nabla\mathcal{R}(\R^{3})
\end{align}
there holds 
\begin{align}\label{eq:KornIIA}
\|F\|_{\lebe^{\frac{3p}{3-p}}(Q)}\leq c\big(\|F^{\sym}\|_{\lebe^{\frac{3p}{3-p}}(Q)}+\|\curl(F)\|_{\lebe^{p}(Q)} \big).
\end{align}
\item\label{item:main1B} For all $F\in\hold^{\infty}(\overline{Q};\R^{3\times 3})$ with 
\begin{align}\label{eq:FIX2}
\int_{Q}\langle F,\Pi\rangle\dif x = 0\qquad\text{for all}\;\Pi\in\nabla\mathcal{K}(\R^{3})
\end{align}
there holds 
\begin{align}\label{eq:KornIIB}
\|F\|_{\lebe^{\frac{3p}{3-p}}(Q)}\leq c\big(\|F^{\mathrm{dev}}\|_{\lebe^{\frac{3p}{3-p}}(Q)}+\|\curl(F)\|_{\lebe^{p}(Q)} \big).
\end{align}
\end{enumerate}
Here, we have set $A^{\mathrm{dev}}:=A^{\sym}-\frac{1}{3}\mathrm{tr}(A)\mathbbm{1}_{3}$ for $A\in\R^{3\times 3}$.  
\end{theorem}
Condition \eqref{eq:FIX1} is particularly transparent, being equivalent to $F^{\mathrm{skew}}$ having integral zero over $Q$. Coming back to our initial discussion, in the framework of \eqref{eq:KornIIA} $F^{\sym}=0$ and $\curl(F)=0$ imply that $F=\nabla u = A$ for some $A\in\mathbb{R}_{\scew}^{3\times 3}$. However, in this situation, the orthogonality condition \eqref{eq:FIX1} with $\Pi=A$ implies $\Pi=0$ and so $F=0$, too. A similar consideration equally yields consistency of inequality \eqref{eq:KornIIB} subject to \eqref{eq:FIX2}. 


The strategy to arrive at Theorem~\ref{thm:main1} is similar to that of Theorem~\ref{thm:main}, where now the global singular integral or Fourier multiplier  estimate underyling \eqref{eq:multiest} is replaced by the Ne\v{c}as-Lions lemma. This strategy has also been pursued in \cite{Neff3,Neff4}, where we now employ a duality estimate as in \cite{BB04A,VS04} to deal with the corresponding negative norms.

In \cite{VS04} (also see \cite{BB04A,BB07}) the following fundamental inequality is established, which moreover can be used to derive the Bourgain-Brezis-estimate for solenoidal fields: There exists a constant $c=c(n)>0$ such that for all $\Phi\in\lebe^{1}(\R^{n};\R^{n})$ with $\mathrm{div}(\Phi)\in\lebe^{1}(\R^{n})$ there holds 
\begin{align}\label{eq:cocancelling}
\int_{\R^{n}}\langle\Phi,\varphi\rangle\dif x \leq c(\|\Phi\|_{\lebe^{1}(\R^{n})}\|\nabla\varphi\|_{\lebe^{n}(\R^{n})}+\|\di(\Phi)\|_{\lebe^{1}(\R^{n})}\|\varphi\|_{\lebe^{n}(\R^{n})} )
\end{align}
for all $\varphi\in\hold_{c}^{\infty}(\R^{n};\R^{n})$. The importance of \eqref{eq:cocancelling} is based on the fact that ${\dot\sobo}{^{1,n}}(\R^{n})\not\hookrightarrow \lebe^{\infty}(\R^{n})$ for $n\geq 2$. To utilise \eqref{eq:cocancelling} in view of Theorem~\ref{thm:main1}, we require a localised version as follows:
\begin{proposition}\label{prop:BBlocal}
There exists a constant $c=c(n)>0$ such that we have 
\begin{align}\label{eq:cocancelling1}
\int_{(0,1)^{n}}\langle\Phi,\varphi\rangle\dif x \leq c\|\Phi\|_{\lebe^{1}((0,1)^{n})}\|\nabla\varphi\|_{\lebe^{n}((0,1)^{n})}
\end{align}
for all $\Phi\in\hold(\overline{(0,1)^{n}};\R^{n})\cap\hold^{1}((0,1)^{n};\R^{n})$ with $\di(\Phi)=0$  and all $\varphi\in\hold_{c}^{\infty}((0,1)^{n};\R^{n})$.
\end{proposition} 
Note that \eqref{eq:cocancelling1} differs from \eqref{eq:cocancelling} (for solenoidal fields) merely by the domain of integration and that of the corresponding Lebesgue norms on the right-hand side. 
\begin{proof}[Proof of Proposition~\ref{prop:BBlocal}]
Let $\Phi$ be as in the proposition. By Lemma~\ref{lem:extendingdivfree} there exists $\widetilde{\Phi}\in\lebe^{1}((-1,2)^{n};\R^{n})$ such that $\widetilde{\Phi}|_{(0,1)^{n}}=\Phi$, $\mathrm{div}(\widetilde{\Phi})=0$ in $\mathscr{D}'((-1,2)^{n};\R^{n})$ and $\|\widetilde{\Phi}\|_{\lebe^{1}((-1,2)^{n})}\leq 3^{n}\|\Phi\|_{\lebe^{1}((0,1)^{n})}$. We pick a smooth cut-off function $\rho\in\hold_{c}^{\infty}((-1,2)^{n})$ with $\mathbbm{1}_{(0,1)^{n}}\leq \rho \leq \mathbbm{1}_{(-1,2)^{n}}$. Then we have, using $\mathrm{div}(\rho\widetilde{\Phi})= \rho\mathrm{div}(\widetilde{\Phi})+\langle\widetilde{\Phi},\nabla\rho\rangle$, 
\begin{align*}
\int_{(0,1)^{n}}\langle\Phi,\varphi\rangle\dif x & = \int_{\R^{n}}\langle\rho\widetilde{\Phi},\varphi\rangle\dif x \\ 
& \!\!\!\stackrel{\eqref{eq:cocancelling}}{\leq} C\big(\|\rho\widetilde{\Phi}\|_{\lebe^{1}(\R^{n})}\|\nabla\varphi\|_{\lebe^{n}(\R^{n})} + \|\mathrm{div}(\rho\widetilde{\Phi})\|_{\lebe^{1}(\R^{n})}\|\varphi\|_{\lebe^{n}(\R^{n})} \big)\\
& \leq C\big(\|\widetilde{\Phi}\|_{\lebe^{1}((-1,2)^{n})}\|\nabla\varphi\|_{\lebe^{n}((0,1)^{n})} + \|\mathrm{div}(\widetilde{\Phi})\|_{\lebe^{1}((-1,2)^{n})}\|\varphi\|_{\lebe^{n}(\R^{n})}\big.\\ & \big.\;\;\;\;\;\;\;\;\;\;\;\;+\|\langle\widetilde{\Phi},\nabla\rho\rangle\|_{\lebe^{1}(\R^{n})}\|\varphi\|_{\lebe^{n}(\R^{n})} \big)\\ 
& \!\!\!\!\!\!\!\!\!\!\!\!\!\!\!\stackrel{\text{Properties of $\widetilde{\Phi}$}}{\leq} C\|\Phi\|_{\lebe^{1}((0,1)^{n})}(\|\varphi\|_{\lebe^{n}((0,1)^{n})}+\|\nabla \varphi\|_{\lebe^{n}(\R^{n})}) \\ 
& \!\leq C\|\Phi\|_{\lebe^{1}((0,1)^{n})}\|\nabla \varphi\|_{\lebe^{n}((0,1)^{n})}, 
\end{align*}
the ultimate inequality being a consequence of Poincar\'{e}'s inequality. This finishes the proof.
\end{proof}
\begin{proof}[Proof of Theorem~\ref{thm:main1}]
Ad~\ref{item:main1A}. Let $1<q<\infty$ and pick $\lebe^{2}(Q;\R^{3\times 3})$-orthonormal bases $\{\mathbf{e}_{1},...,\mathbf{e}_{m_{1}}\}$, $\{\mathbf{f}_{1},...,\mathbf{f}_{m_{2}}\}$ of $\nabla\mathcal{R}(\R^{3})$ or $\nabla\mathcal{K}(\R^{3})$, respectively.  We then record from \cite[Eq.~(42)]{Neff3} and \cite[Eq.~(3.24)]{Neff4} that there exists a constant $c=c(q)>0$ such that\footnote{In the argument underlying \cite[Eq.~(42)]{Neff3} (and similarly for \cite[Eq.~(3.24)]{Neff4}), the authors deal with functionals $l_{i}\colon\nabla\mathcal{R}(\R^{3})\to\R$ or $l_{i}\colon\nabla\mathcal{K}(\R^{3})\to\R$ which satisfy $l_{i}(\mathbf{e}_{j})=\delta_{ij}$ and extend them to $\sobo^{\curl,p}(Q;\R^{3\times 3})$ by Hahn-Banach. The choices $F\mapsto \langle\mathbf{e}_{i},F\rangle_{\lebe^{2}}$ or $F\mapsto\langle\mathbf{f}_{i},F\rangle_{\lebe^{2}}$ can, since $\mathbf{e}_{j},\mathbf{f}_{j}\in\lebe^{\infty}$, directly be defined on $\sobo^{\curl,p}(Q;\R^{3\times 3})$ without appealing to Hahn-Banach.}
\begin{align}\label{eq:symcurlpreliminary}
&\|F\|_{\lebe^{q}(Q)}\leq c\Big(\|F^{\sym}\|_{\lebe^{q}(Q)}+\|\curl(F)\|_{\sobo^{-1,q}(Q)}+\sum_{\ell=1}^{m_{1}}\left\vert\int_{Q}\langle\mathbf{e}_{\ell},F\rangle\dif x\right\vert\Big),\\\label{eq:devsymcurlpreliminary}
&\|F\|_{\lebe^{q}(Q)}\leq c\Big(\|F^{\mathrm{dev}}\|_{\lebe^{q}(Q)}+\|\curl(F)\|_{\sobo^{-1,q}(Q)}+\sum_{\ell=1}^{m_{2}}\left\vert\int_{Q}\langle\mathbf{f}_{\ell},F\rangle\dif x\right\vert\Big)
\end{align}
hold for all $F\in\hold^{\infty}(\overline{Q};\R^{3\times 3})$. It is precisely these estimates which are a consequence of the Ne\v{c}as-Lions lemma. We now distinguish two cases: 

\emph{Case $1<p<3$.} We note that $(\frac{3p}{3-p})'=\frac{3p}{4p-3}$ and, since $1<p<3$, $\frac{3p}{4p-3}\in (1,3)$. Thus, by the usual Sobolev embedding theorem and denoting $\theta^{*}=\frac{3\theta}{3-\theta}$,  
\begin{align}\label{eq:auxilemb}
{\dot\sobo}^{1,\frac{3p}{4p-3}}(\R^{3};\R^{3\times 3}) \hookrightarrow \lebe^{(\frac{3p}{4p-3})^{*}}(\R^{3};\R^{3\times 3}) = \lebe^{p'}(\R^{3};\R^{3\times 3}). 
\end{align}
Therefore, by H\"{o}lder's inequality, 
\begin{align}\label{eq:symcurlest}
\begin{split}
\|\curl(F)\|_{\sobo^{-1,\frac{3p}{3-p}}(Q)} & =  \sup_{\substack{\varphi\in\hold_{c}^{\infty}(Q;\R^{3\times 3})\\ \|\nabla\varphi\|_{\lebe^{\frac{3p}{4p-3}}(Q;\R^{3\times 3})}\leq 1}} \int_{Q}\curl(F)\cdot\varphi\dif x \\ 
& \leq \sup_{\substack{\varphi\in\hold_{c}^{\infty}(Q;\R^{3\times 3})\\ \|\nabla\varphi\|_{\lebe^{\frac{3p}{4p-3}}(Q;\R^{3\times 3})}\leq 1}} \|\curl(F)\|_{\lebe^{p}(Q)}\|\varphi\|_{\lebe^{p'}(Q)} \\ 
& \!\!\!\stackrel{\eqref{eq:auxilemb}}{\leq} c \|\curl(F)\|_{\lebe^{p}(Q)}. 
\end{split}
\end{align}
Now, combining \eqref{eq:symcurlpreliminary} and \eqref{eq:symcurlest} with $q=\frac{3p}{3-p}$, we obtain Theorem~\ref{thm:main1}~\ref{item:main1A} for $1<p<3$ by virtue of \eqref{eq:FIX1}. To arrive at Theorem~\ref{thm:main1}~\ref{item:main1B} for $1<p<3$, we argue analogously but now using \eqref{eq:devsymcurlpreliminary}.

\emph{Case $p=1$.} We only have to establish \eqref{eq:symcurlest} for $p=1$. We apply Proposition~\ref{prop:BBlocal} to $\Phi^{i}:=\curl(F^{i})$ so that $\di(\Phi^{i})=0$ for $i\in\{1,2,3\}$. Therefore, 
\begin{align*}
\|\curl(F)\|_{\sobo^{-1,\frac{3}{2}}(Q)} & = \sup_{\substack{\varphi\in\hold_{c}^{\infty}(Q;\R^{3\times 3})\\ \|\nabla\varphi\|_{\lebe^{3}(Q;\R^{3\times 3})}\leq 1}}\int_{Q}\langle\curl(F),\varphi\rangle\dif x \\ 
& = \sum_{i\in\{1,2,3\}}\sup_{\substack{\varphi\in\hold_{c}^{\infty}(Q;\R^{3\times 3})\\ \|\nabla\varphi\|_{\lebe^{3}(Q;\R^{3\times 3})}\leq 1}}\int_{Q}\langle \curl(F^{i}),\varphi^{i} \rangle\dif x \\ 
& \leq c\sup_{\substack{\varphi\in\hold_{c}^{\infty}(Q;\R^{3\times 3})\\ \|\nabla\varphi\|_{\lebe^{3}(Q;\R^{3\times 3})}\leq 1}}\Big(\sum_{i\in\{1,2,3\}}\|\curl(F^{i})\|_{\lebe^{1}(Q)}\Big)\|\nabla\varphi\|_{\lebe^{3}(Q)} \\ 
& \leq c\|\curl(F)\|_{\lebe^{1}(Q)}.
\end{align*}
The proof is hereby complete. 
\end{proof}
\begin{remark}
If $1<p<3$, then the above proof shows that we may replace the unit cube $Q$ by any open and bounded, simply connected domain $\Omega$ with Lipschitz boundary. 
\end{remark}
We conclude the paper by addressing possible generalisations of Theorem~\ref{thm:main1}:
\begin{openquestion}[On more general operators and domains]\label{OQ:1}
\noindent
(a) Following the argument of \cite{Neff3,Neff4} (in particular \cite[Cor.~2.3]{Neff3}), if $\A$ is an operator of the form \eqref{eq:form} with $\R^{N}=\R^{9}\cong\R^{3\times 3}$, the Ne\v{c}as-Lions lemma (\emph{cf.} \cite[Thm.~1]{Necas}) can be utilised to derive the respective variant of \eqref{eq:KornIIA} or \eqref{eq:KornIIB} provided $\dim(\ker(\A))<\infty$ and there exists $m\in\mathbb{N}_{\geq 1}$ and a linear map $\mathcal{L}\colon\odot^{m-1}(\R^{3};\R^{3\times 3})\to\odot^{m}(\R^{3};\R^{3\times 3})$ such that 
\begin{align}\label{eq:Celliptic}\tag{$*$}
D^{m}F=\mathcal{L}(D^{m-1}\curl(F))\qquad\text{for all}\;F\in\hold^{m}(Q;(\mathrm{Id}-A)(\R^{3\times 3})). 
\end{align}
This suggests that Theorem~\ref{thm:main1} should be generalisable to the class of $\mathbb{C}$-elliptic differential operators (\emph{cf.}~\cite{Smith,Kalamajska,BDG}) as the finite dimensionality of the nullspace is the characterising feature of such differential operators, but it is not clear to us how to establish \eqref{eq:Celliptic} for this class of operators.\\
(b) If $\Omega\subset\R^{3}$ is an open, bounded, simply connected Lipschitz domain, then estimates \eqref{eq:symcurlpreliminary} and \eqref{eq:devsymcurlpreliminary} persist. To obtain Theorem~\ref{thm:main1}, the above approach works analogously provided one can establish an extension operator $\mathcal{E}\colon \hold(\overline{\Omega};\R^{n})\cap\hold^{1}(\Omega;\R^{n})\to \lebe^{1}(U;\R^{n})$, where $U\subset\R^{n}$ is open with $\Omega\Subset U$, such that 
\begin{align}\label{eq:extension}\tag{$**$}
\begin{cases} 
(\mathcal{E}\Phi)|_{\Omega}=\Phi, \\ 
\mathrm{div}(\Phi)=0\;\text{in}\;\Omega\Longrightarrow \mathrm{div}(\mathcal{E}\Phi)=0\;\text{in}\;\mathscr{D}'(U),\\ 
\|\mathcal{E}\Phi\|_{\lebe^{1}(U)}\leq c\|\Phi\|_{\lebe^{1}(\Omega)}
\end{cases} 
\end{align}
for some $c>0$ and all $\Phi\in \hold(\overline{\Omega};\R^{n})\cap\hold^{1}(\Omega;\R^{n})$. Note that the usual extension techniques hinging on localisation by means of partitions of unity (\emph{cf.} \cite[Chpt.~4.4]{EG})  destroy the solenoidality of the extensions. In consequence, it would be of interest to know whether any open, bounded and simply connected Lipschitz domain $\Omega\subset\R^{3}$ admits an extension operator $\mathcal{E}$ satisfying \eqref{eq:extension} for some open set $U$ with $\Omega\Subset U$. 
\end{openquestion}

\section*{Appendix}
Although the following extension result underlying the proof of Proposition~\ref{prop:BBlocal} should be well-known to the experts, it is hard to be traced back in the literature and so we state and give the quick proof here:
\begin{lemma}\label{lem:extendingdivfree}
There exists a linear extension operator $\mathcal{E}\colon\hold(\overline{(0,1)^{n}};\R^{n})\cap\hold^{1}((0,1)^{n};\R^{n})\to\lebe^{1}((-1,2)^{n};\R^{n})$ such that the following hold for all $\Phi\in\hold(\overline{(0,1)^{n}};\R^{n})\cap\hold^{1}((0,1)^{n};\R^{n})$:
\begin{align}\label{eq:extension1}
\begin{cases} 
(\mathcal{E}\Phi)|_{(0,1)^{n}}=\Phi, \\ 
\mathrm{div}(\Phi)=0\;\text{in}\;(0,1)^{n}\Longrightarrow \mathrm{div}(\mathcal{E}\Phi)=0\;\text{in}\;\mathscr{D}'((-1,2)^{n}),\\ 
\|\mathcal{E}\Phi\|_{\lebe^{1}((-1,2)^{n})}\leq 3^{n}\|\Phi\|_{\lebe^{1}((0,1)^{n})}.
\end{cases} 
\end{align}
\end{lemma}
\begin{proof} 
We proceed by induction over the first $k$ elements of $\{1,...,n-1\}$. Suppose that $\Psi$ is defined on $(-1,2)^{k-1}\times(0,1)^{n-k+1}$. We claim that there exists an operator $\mathcal{E}_{k}$ with $\mathcal{E}_{k}\Psi\colon (-1,2)^{k}\times (0,1)^{n-k}\to\R^{n}$ such that
\begin{align}\label{eq:extension2}
\begin{cases} 
(\mathcal{E}_{k}\Psi)|_{(-1,2)^{k-1}\times (0,1)^{n-k+1}}=\Psi, \\ 
\mathrm{div}(\Psi)=0\;\text{in}\;(-1,2)^{k-1}\times (0,1)^{n-k+1} \\ \;\;\;\;\;\;\;\;\;\;\;\;\;\;\;\;\;\;\;\;\;\;\;\;\;\Longrightarrow \mathrm{div}(\mathcal{E}_{k}\Psi)=0\;\text{in}\;\mathscr{D}'((0,1)^{k}\times (0,1)^{n-k}),\\ 
\|\mathcal{E}_{k}\Psi\|_{\lebe^{1}((-1,2)^{k}\times (0,1)^{n-k})}\leq 3\|\Psi\|_{\lebe^{1}((-1,2)^{k-1}\times(0,1)^{n-k+1})}, 
\end{cases} 
\end{align}
where we adopt the convention $(-1,2)^{0}\times (0,1)^{n}=(0,1)^{n}$. Then, by construction, $\mathcal{E}:=\mathcal{E}_{n}\circ\mathcal{E}_{n-1}\circ...\circ\mathcal{E}_{1}$ satisfies \eqref{eq:extension1}. For $k\in\{1,...,n-1\}$, define for $\Psi\colon (-1,2)^{k-1}\times (0,1)^{n-k+1}\to\R^{n}$
\begin{align*}
\mathcal{E}_{k}\Psi (x) :=\begin{cases} 
\mathcal{E}_{k}^{+}\Psi(x)&\;\text{if}\; x\in (-1,2)^{k-1}\times (1,2)\times (0,1)^{n-k}\\ 
\Psi(x) &\;\text{if}\;x\in(-1,2)^{k-1}\times (0,1)\times (0,1)^{n-k}\\ 
\mathcal{E}_{k}^{-}\Psi(x)&\;\text{if}\; x\in (-1,2)^{k-1}\times (-1,0)\times (0,1)^{n-k}, 
\end{cases} 
\end{align*}
where $\mathcal{E}_{k}^{\pm}\Psi=((\mathcal{E}_{k}^{\pm}\Psi)_{1},...,(\mathcal{E}_{k}^{\pm}\Psi)_{n})$ with 
\begin{align*}
&\begin{cases} 
(\mathcal{E}_{k}^{+}\Psi)_{j}(x_{1},...,x_{n}):= - \Psi_{j}(x_{1},...,x_{k-1},2-x_{k},x_{k+1},...,x_{n-1},x_{n})&\;\text{if}\;j\neq k,\\
(\mathcal{E}_{k}^{+}\Psi)_{k}(x_{1},...,x_{n}):=\Psi_{k}(x_{1},...,x_{k-1},2-x_{k},x_{k+1},...,x_{n})&\;\text{if}\;j=k
\end{cases}\\ & \;\;\;\;\;\;\;\;\;\;\;\;\;\;\;\;\;\;\;\;\;\;\;\;\;\;\;\;\;\;\;\;\;\;\;\;\;\;\;\;\;\;\;\;\;\;\;\;\;\;\;\;\;\;\;\;\;\;\;\;\;\;\;\;\;\;\;\;\;\;\;\;\;\;\text{for}\; x\in (1,2)^{k-1}\times(1,2)\times (0,1)^{n-k},\\
&\begin{cases} 
(\mathcal{E}_{k}^{-}\Psi)_{j}(x_{1},...,x_{n}):= - \Psi_{j}(x_{1},...,x_{k-1},-x_{k},x_{k+1},...,x_{n-1},x_{n})&\;\text{if}\;j\neq k,\\
(\mathcal{E}_{k}^{-}\Psi)_{k}(x_{1},...,x_{n}):=\Psi_{k}(x_{1},...,x_{k-1},-x_{k},x_{k+1},...,x_{n})&\;\text{if}\;j=k
\end{cases}\\
& \;\;\;\;\;\;\;\;\;\;\;\;\;\;\;\;\;\;\;\;\;\;\;\;\;\;\;\;\;\;\;\;\;\;\;\;\;\;\;\;\;\;\;\;\;\;\;\;\;\;\;\;\;\;\;\;\;\;\;\;\;\;\;\;\;\;\;\;\;\;\;\;\;\;\text{for}\; x\in (1,2)^{k-1}\times(-1,0)\times (0,1)^{n-k}.
\end{align*}
By construction,~$\eqref{eq:extension2}_{1}$ and $\eqref{eq:extension2}_{3}$ follow. Let $\varphi\in\hold_{c}^{\infty}((-1,2)^{k}\times (0,1)^{n-k})$. Since $\di(\mathcal{E}_{k}^{\pm}\Psi)=0$ on $(1,2)^{k-1}\times (1,2)\times (0,1)^{n-k}$ or $(1,2)^{k-1}\times (-1,0)\times(0,1)^{n-k}$, respectively, we have with $\nu=(\nu_{i})_{i}=(0,...,0,-1,0,...,0)$ and $\widetilde{\nu}=(\widetilde{\nu}_{i})_{i}=(0,...,0,1,0,...,0)$ (the non-zero entry sitting at the $k$-th position) 
\begin{align*}
\int_{(-1,2)^{k}\times (0,1)^{n-k}}&\langle\mathcal{E}_{k}\Psi,\nabla\varphi\rangle\dif x  = -\sum_{j\in\{1,...,n\}\setminus\{k\}}\Big(\int_{(-1,2)^{k-1}\times\{0\}\times (0,1)^{n-k}} (\mathcal{E}_{k}^{-}\Psi)_{j}\varphi\nu_{j}\dif\mathscr{H}^{n-1}\Big.\\ 
& \Big. - \int_{(-1,2)^{k-1}\times\{0\}\times (0,1)^{n-k}}\Psi_{j}\varphi\nu_{j}\dif\mathscr{H}^{n-1}\Big) \\ & -\sum_{j\in\{1,...,n\}\setminus\{k\}}\Big(\int_{(-1,2)^{k-1}\times\{1\}\times (0,1)^{n-k}} (\mathcal{E}_{k}^{+}\Psi)_{j}\varphi\widetilde{\nu}_{j}\dif\mathscr{H}^{n-1}\Big.\\ 
& \Big. - \int_{(-1,2)^{k-1}\times\{1\}\times (0,1)^{n-k}}\Psi_{j}\varphi\widetilde{\nu}_{j}\dif\mathscr{H}^{n-1}\Big) \\ & \mp\Big(\int_{(-1,2)^{k-1}\times\{0\}\times (0,1)^{n-k}} \Psi_{k}\varphi\nu_{k}\dif\mathscr{H}^{n-1}\Big.\\ 
& \Big. + \int_{(-1,2)^{k-1}\times\{1\}\times (0,1)^{n-k}}\Psi_{k}\varphi\widetilde{\nu}_{k}\dif\mathscr{H}^{n-1}\Big)= 0
\end{align*}
as only those summands with $j=k$ are potentially non-zero, and for $j=k$ the corresponding integrals cancel out. Hence $\di(\mathcal{E}_{k}\Psi)=0$ in $\mathscr{D}'((-1,2)^{k}\times(0,1)^{n-k})$. This finishes the proof. 
\end{proof}

\end{document}